\documentclass[reqno, 11pt]{amsart}
\usepackage{amsthm, amsmath, amssymb, graphicx, tikz,enumerate,paralist,bbm,setspace,bbm}
\usepackage[left=1in,right=1in,top=1in,bottom=1in]{geometry}

\usepackage[color,final]{showkeys} 
\usepackage[colorlinks=false, pdfstartview=FitV, 
pagebackref=false,hidelinks]{hyperref}

\definecolor{refkey}{gray}{.75}
\definecolor{labelkey}{gray}{.5}

\usetikzlibrary{calc,intersections,through,backgrounds,shapes,patterns}
\tikzstyle{p}+=[fill=black, circle, minimum width = 1pt, inner sep =
1pt]

\newtheorem{theorem}{Theorem}[section]

\newtheorem{lemma}[theorem]{Lemma}

\theoremstyle{definition}
\newtheorem{definition}[theorem]{Definition}
\newtheorem{example}[theorem]{Example}

\theoremstyle{remark}

\newcommand{\abs}[1]{\left\lvert#1\right\rvert}

\newcommand{\norm}[1]{\left\lVert#1\right\rVert}

\newcommand{\x}{\times}
\newcommand{\e}{\epsilon}

\newcommand{\EE}{\mathbb{E}}
\newcommand{\RR}{\mathbb{R}}

\newcommand{\ZZ}{\mathbb{Z}}

\newcommand{\cF}{\mathcal{F}}

\newcommand{\tg}{\tilde{g}}
\newcommand{\tf}{\tilde{f}}

\title[An arithmetic transference proof of a relative Szemer\'edi theorem]{An arithmetic transference proof of \\ a relative Szemer\'edi theorem}
\author{Yufei Zhao}
\address{Department of Mathematics\\
MIT\\
Cambridge\\
MA 02139-4307}
\email{yufeiz@math.mit.edu}

\thanks{This work was done while the author was an intern at Microsoft
  Research New England.}

\begin{document}

\maketitle

\begin{abstract}
  Recently, Conlon, Fox, and the author gave a new proof of a relative
  Szemer\'edi theorem, which was the main novel ingredient in the proof of
  the celebrated Green-Tao theorem that the primes contain arbitrarily
  long arithmetic progressions. Roughly speaking, a relative
  Szemer\'edi theorem says that if $S$ is a set of integers satisfying
  certain conditions, and $A$ is a subset of $S$ with positive relative
  density, then $A$ contains long arithmetic progressions, and our
  recent results show that $S$ only needs to satisfy a so-called
  linear forms condition.

  This note contains an alternative proof of the new relative Szemer\'edi
  theorem, where we directly transfer Szemer\'edi's theorem, instead of
  going through the hypergraph removal lemma. This approach provides a
  somewhat more direct route to establishing the result, and it gives better quantitative bounds.

  The proof has three main ingredients: (1) a transference
  principle/dense model theorem of Green-Tao and Tao-Ziegler (with
  simplified proofs given later by Gowers, and independently,
  Reingold-Trevisan-Tulsiani-Vadhan) applied with a
  discrepancy/cut-type norm (instead of a Gowers uniformity norm as it
  was applied in earlier works), (2) a counting lemma
  established by Conlon, Fox, and the author, and (3) Szemer\'edi's
  theorem as a black box.
\end{abstract}

\section{Introduction}

The celebrated Green-Tao theorem~\cite{GT08} states that the primes
contain arbitrarily long arithmetic progressions (AP). A key ingredient
in their work is a relative Szemer\'edi theorem. Szemer\'edi's theorem~\cite{Sze75}
states that any subset of the integers with positive upper density
contains arbitrarily long arithmetic progressions. A relative
Szemer\'edi theorem is a result where the ground set is no longer $\ZZ$ but some sparse pseudorandom subset (or more generally some
measure).

Green and Tao proved a relative Szemer\'edi theorem provided that the ground
set satisfies certain pseudorandomness conditions known as the linear
forms condition and the correlation condition. They then constructed
a majorizing measure to the primes, using ideas from the work of Goldston
and Y{\i}ld{\i}r{\i}m~\cite{GY03} (subsequently simplified
in~\cite{Taonote}), so that this majorizing measure satisfies the desired pseudorandomness conditions.

Recently, Conlon, Fox, and the author~\cite{CFZrelsz} gave a new proof
of Green and Tao's relative Szemer\'edi theorem, requiring simpler
pseudorandomness hypotheses on the ground set. We showed that a weak
version of Green and Tao's linear forms condition is sufficient. A
precise definition of our linear forms condition will be given in \S\ref{sec:def}.

In \cite{CFZrelsz}, we obtained a relative extension of the famous
hypergraph removal lemma \cite{Gow07, NRS06, RS04, RS06, Tao06jcta},
from which we deduced our relative Szemer\'edi theorem via standard
arguments. Such an approach was also taken by Tao~\cite{Tao06jam} in
his work on
constellations in the Gaussian primes, but our approach in \cite{CFZrelsz}
requires less stringent pseudorandomness hypotheses.

In this note, we give an alternative approach to proving the relative
Szemer\'edi theorem in \cite{CFZrelsz}. Instead of going through the
hypergraph removal lemma, we use Szemer\'edi's theorem directly as a
black box. To transfer Szemer\'edi's theorem to the sparse setting, we
apply the dense model theorem of Green-Tao~\cite{GT08} and
Tao-Zieger~\cite{TZ08}, which was subsequently simplified by
Gowers~\cite{Gow10}, and independently Reingold, Trevisan, Tulsiani,
and Vadhan \cite{RTTV08}. This tool lets us model a subset of a sparse
pseudorandom set of integers by a dense subset. The dense model is a
good approximation of the original set with respect to a
discrepancy-type norm (similar to the cut metric for graphs). This
contrasts previous proofs the Green-Tao theorem
\cite{GT08,Gow10,RTTV08} where the dense model theorem is applied with
respect to the Gowers uniformity norm, which gives a stronger notion of
approximation.

Another important ingredient in the proof
is the relative counting lemma of \cite{CFZrelsz}, which implies that the
dense model behaves similarly to the original set in the number of
arithmetic progressions.

The arithmetic transference approach presented here
 establishes the new relative Szemer\'edi theorem in a more direct fashion. It also
gives better quantitative bounds than \cite{CFZrelsz}. Indeed, instead of going through
the hypergraph removal lemma, which currently has an Ackermann-type dependence
on the bounds (due to the application of the hypergraph regularity lemma),
we can now use Szemer\'edi's theorem as a black box and
automatically transfer the best quantitative bounds available (currently
the state-of-art is \cite{San11} for 3-term APs, \cite{GT09r4} for 4-term APs, and
\cite{Gow01} for longer APs). This answers a question that was
left open in \cite{CFZrelsz}. The approach
presented here, however, is less general compared to \cite{CFZrelsz},
since it does not provide a relative hypergraph removal lemma, nor
does it give a more general sparse regularity approach to
hypergraphs.

\medskip

The main theorem is stated in \S\ref{sec:def}. In
\S\ref{sec:trans}, we apply a dense model theorem to find a dense
approximation of the original set. In \S\ref{sec:count}, we apply a
counting lemma to show that the dense model has approximately the same
number of $k$-term APs as the original set. Finally in
\S\ref{sec:relsz}, we put everything together and apply
Szemer\'edi's theorem as a black box to conclude the proof.

\section{Definitions and results} \label{sec:def}

\subsection*{Notation} \emph{Dependence on $N$.} We consider functions
$\nu = \nu^{(N)}$, where $N$ (usually suppressed) is assumed to be
some large integer. We write $o(1)$ for a quantity that tends to zero
as $N \to \infty$ along some subset of $\ZZ$. If the rate at which the
quantity tends to zero depends on some other parameters (e.g., $k,
\delta$), then we put these parameters in the subscript (e.g., $o_{k,\delta}(1)$).

\noindent
\emph{Expectation.}  We write $\EE[f(x_1, x_2, \dots) \vert P]$ for
the expectation of $f(x_1, x_2, \dots)$ when the variables are chosen
uniformly out of all possibilities satisfying $P$.

\medskip

We shall use, as a black box, the following weighted version of
Szemer\'edi's theorem as formulated, for example, in
\cite[Prop.~2.3]{GT08}. It may be helpful to think of $f$ as the
indicator function $1_A$ of some set $A \subseteq \ZZ_N$. It will be easier for us to work in $\ZZ_N :=
\ZZ/N\ZZ$ as opposed to $[N] := \{1, \dots, N\}$, although these two
settings are easily seen to be equivalent.

\begin{theorem}
  [Szemer\'edi's theorem, weighted version]
  \label{thm:sz-weighted}
  Let $k \geq 3$ and $0 < \delta \leq 1$ be fixed. Let $f \colon \ZZ_N \to [0,1]$ be a
  function satisfying
  $\EE[f] \geq \delta$. Then
  \begin{equation} \label{eq:sz-weighted}
  \EE[ f(x)f(x+d)f(x + 2d) \cdots f(x+(k-1)d) \vert x,d \in \ZZ_N]
  \geq c(k,\delta) - o_{k,\delta}(1)
\end{equation}
for some constant $c(k,\delta) > 0$ which does not depend on $f$ or
$N$.
\end{theorem}

Gowers' results \cite{Gow01} (along with a Varnavides-type
\cite{Var59} averaging argument) imply that
Theorem~\ref{thm:sz-weighted} holds with $c(k,\delta) =
\exp(-\exp(\delta^{-c_k}))$ with $c_k = 2^{2^{k+9}}$ (see
\cite{San11} and \cite{GT09r4} for the current best bounds for $k=3$ and $4$ respectively).

A \emph{relative Szemer\'edi theorem} is an extension of
Theorem~\ref{thm:sz-weighted} of the following form. Instead of $0 \leq f \leq 1$ as in
Theorem~\ref{thm:sz-weighted}, we now assume that $0 \leq f(x) \leq
\nu(x)$ for all $x \in \ZZ_N$, where $\nu \colon \ZZ_N \to \RR_{\geq 0}$ is some function
(also called a majorizing measure) that satisfies certain
pseuodorandomness conditions. Here the function $\nu$ is normalized so
that $\EE[\nu] = 1 + o(1)$. For instance, one can think of $\nu$ as
$\frac{N}{\abs{S}}1_S$ for some pseudorandom subset $S \subseteq
\ZZ_N$, and $f$ as $1_A \nu$ with some $A \subseteq S$. So in this case
\eqref{eq:sz-weighted} says that $A$ contains many $k$-term APs when
$N$ is sufficiently large.

As in \cite{CFZrelsz}, the pseudorandomness condition that we assume
on $\nu$ is the linear forms condition, as follows.

\begin{definition}[Linear forms condition]
  \label{def:Z-lfc}
  A nonnegative function $\nu = \nu^{(N)} : \mathbb{Z}_N
  \rightarrow \mathbb{R}_{\geq 0}$ is said to obey the
  \emph{$k$-linear forms condition} if one has
  \begin{equation}
    \label{eq:Z-lfc}
    \EE \Bigl[
    \prod_{j=1}^k \prod_{\omega \in \{0,1\}^{[k]\setminus\{j\}}} \nu
    \Bigl( \sum_{i=1}^k (i-j) x_i^{(\omega_i)} \Bigr)^{n_{j,\omega}}
    \Big\vert
    x_1^{(0)}, x_1^{(1)}, \dots, x_k^{(0)},x_k^{(1)} \in \ZZ_N
    \Bigr]  = 1 + o(1)
  \end{equation}
  for any choice of exponents $n_{j,\omega} \in \{0,1\}$.
\end{definition}

\begin{example}
  \label{ex:Z-lfc}
  For $k = 3$, condition~\eqref{eq:Z-lfc} says that
  \begin{multline*}
    \EE[\nu(y+2z) \nu(y'+2z) \nu(y+2z') \nu(y'+2z')
    \nu(-x+z) \nu(-x'+z) \nu(-x+z') \nu(-x'+z') \\
    \nu(-2x-y) \nu(-2x'-y) \nu(-2x-y') \nu(-2x'-y')
    \vert x,x',y,y',z,z' \in \ZZ_N] = 1 + o(1)
  \end{multline*}
  and similar conditions hold if one or more of the twelve $\nu$ factors in the
  expectation are erased.
\end{example}

The main result of this note is the following theorem.

\begin{theorem}
  [Relative Szemer\'edi theorem]
  \label{thm:rel-sz}
  Let $k \geq 3$ and $0 < \delta \leq 1$ be fixed. Let $\nu \colon \ZZ_N \to \RR_{\geq 0}$ satisfy the $k$-linear
  forms condition. Assume that $N$ is sufficiently large and
  relatively prime to $(k-1)!$. Let $f \colon \ZZ_N \to \RR_{\geq 0}$ satisfy $0
  \leq f(x) \leq \nu(x)$ for all $x \in \ZZ_N$ and $\EE[f] \geq
  \delta$. Then
  \begin{equation} \label{eq:rel-sz}
  \EE[ f(x)f(x+d)f(x + 2d) \cdots f(x+(k-1)d) \vert x,d \in \ZZ_N]
  \geq c(k,\delta) - o_{k,\delta}(1),
\end{equation}
where $c(k,\delta)$ is the same constant which appears in
Theorem~\ref{thm:sz-weighted}. The rate at which the $o_{k,\delta}(1)$
term goes to zero
depends not only on $k$ and $\delta$ but also the rate of convergence in the $k$-linear forms condition for $\nu$.
\end{theorem}

This theorem was proved in \cite{CFZrelsz} without the additional
conclusion that $c(k,\delta)$ can be taken to be the same as in
Theorem~\ref{thm:sz-weighted}. Indeed, the proof in \cite{CFZrelsz}
uses the hypergraph removal lemma as a black box, so that the
constants $c(k,\delta)$ there are much worse, with an Ackermann-type
dependence due to the use of hypergraph regularity. In \cite{GT08},
Green and Tao also transfered Szemer\'edi's theorem directly to obtain the same constants
$c(k,\delta)$ as in Theorem~\ref{thm:sz-weighted}, but under stronger pseudorandomness hypotheses for
$\nu$. So Theorem~\ref{thm:rel-sz} combines the conclusions of the two relative Szemer\'edi theorems in  \cite{CFZrelsz} and \cite{GT08}.

\section{Dense model theorem} \label{sec:trans}

In this section, we show that the $f$ in Theorem~\ref{thm:rel-sz} can
be modeled by a function $\tf \colon \ZZ_N \to [0,1]$. We state our
results in terms of a finite abelian group $G$ (written additively),
but there is no loss in thinking $G = \ZZ_N$. For $x = (x_1, \dots, x_r) \in G^r$, and $I \subseteq [r]$, we write $x_I = (x_i)_{i \in I}$.

\begin{definition}
  Let $G$ be a finite abelian group, $r$ be a positive integer,
  $\psi \colon G^r \to G$ be a surjective homomorphism, and $f, \tf \colon G \to \RR_{\geq 0}$ be two functions. We say
  that $(f,\tf)$ is an \emph{$(r,\e)$-discrepancy pair with respect to
    $\psi$} if
\begin{equation} \label{eq:disc-pair}
  \Bigl\lvert \EE\Bigl[ (f(\psi(x)) - \tf (\psi(x))) \prod_{i=1}^r u_i(x_{[r]\setminus\{i\}})
  \Big\vert
  x \in G^r
  \Bigr]\Bigr\rvert \leq \e
\end{equation}
for all collections of functions $u_1, \dots, u_r \colon G^{r-1}
  \to [0,1]$.
\end{definition}

\begin{example}
When $r = 2$ and $\psi(x,y) = x+y$, \eqref{eq:disc-pair} says
\[
\lvert \EE[(f(x+y) - \tf(x+y)) u_1(y)u_2(x) | x,y \in G] \rvert \leq \e.
\]
In other words, this says that the two weighted graphs $g,\tg \colon G \x G \to \RR_{\geq 0}$ given by $g(x,y) = f(x+y)$ and $\tg(x,y) = \tf(x+y)$ satisfy $\norm{g-\tg}_{\square} \leq \e$, where $\norm{\cdot}_{\square}$ is the cut norm for bipartite graphs.

When $r = 3$ and $\psi(x,y,z) = x+y+z$, \eqref{eq:disc-pair} says
\[
\lvert \EE[(f(x+y+z) - \tf(x+y+z)) u_1(y,z)u_2(x,z)u_3(x,y) | x,y,z \in G] \rvert \leq \e.
\]
\end{example}

The following key lemma says that any $0 \leq f \leq \nu$ can be approximated by a $0 \leq \tf \leq 1$ in the above sense.

\begin{lemma}
  \label{lem:disc-pair-transference}
  For every $\e > 0$ there is an $\e' = \exp(-\e^{-O(1)})$ such
  that the following holds:

  Let $G$ be a finite abelian group, $r$ be a positive integer, and
  $\psi \colon G^r \to G$ be a surjective homomorphism. Let $f, \nu
  \colon G \to \RR_{\geq 0}$ be such that $0 \leq f \leq
  \nu$, $\EE[f] \leq 1$, and $(\nu, 1)$ is an $(r,\e')$-discrepancy pair with respect to
  $\psi$. Then there exists a function $\tf \colon G \to [0,1]$ so that $\EE[\tf] = \EE[f]$ and $(f,\tf)$ is an $(r,\e)$-discrepancy pair with respect to
  $\psi$.
\end{lemma}

The proof of Lemma~\ref{lem:disc-pair-transference} uses the
dense model theorem of Green-Tao~\cite{GT08} and
Tao-Ziegler~\cite{TZ08}, which was later simplified
in \cite{Gow10} and \cite{RTTV08}. The expository note \cite{RTTVnote}
has a nice and short write-up of the proof of the dense model theorem,
and we quote the statement from there.

Let $X$ be a finite set. For any two functions $f,g \colon X \to \RR$, we write $\langle f,g \rangle = \EE[f(x)g(x) \vert x \in X]$.
For $\cF$ a collection of functions $\varphi \colon X \to
[-1,1]$, we write $\cF^k$ to mean the collections of all functions of the form
$\prod_{i=1}^{k'} \varphi_i$, where $\varphi_i \in \cF$ and $k' \leq k$. In
particular, if $\cF$ is closed under multiplication, then $\cF^k = \cF$.

\begin{lemma} [Green-Tao-Ziegler dense model theorem]
  \label{lem:GTZ-DMT}
  For every $\e > 0$, there is a $k = (1/\e)^{O(1)}$ and an $\e' =
  \exp(-(1/\e)^{O(1)})$ such that the following holds:

  Suppose that $\cF$ is a collection of functions $\varphi
  \colon X \to [-1,1]$ on a finite set $X$, $\nu \colon X \to \RR_{\geq 0}$
  satisfies
  \[
  \lvert \langle \nu - 1, \varphi \rangle\vert \leq \e' \text{ for all
  } \varphi \in \cF^k,
  \]
  and $f \colon X
  \to \RR_{\geq 0}$ satisfies $f \leq \nu$ and $\EE[f] \leq 1$.
  Then there is a function $\tf \colon X \to [0,1]$ such that
  $\EE[\tf] = \EE[f]$, and
  \[
  \lvert\langle f - \tf, \varphi \rangle\vert \leq \e \text{ for all }
  \varphi \in \cF.
  \]
\end{lemma}

We shall use Lemma~\ref{lem:GTZ-DMT} with $\cF$  closed
under multiplication, so that $k$ plays no role. This is an important
point in our simplification over previous approaches using the dense
model theorem.

\begin{proof}[Proof of Lemma~\ref{lem:disc-pair-transference}]
  For any collection of functions $u_1, \dots, u_r \colon G^{r-1}
  \to \RR$, define a generalized convolution $(u_1, \dots, u_r)^*_\psi
  \colon G \to \RR$ by
  \[
  (u_1, \dots, u_r)^*_\psi(x)
  =
  \EE\Bigl[ \prod_{i=1}^r u_i(y_{[r]\setminus\{i\}})
  \Big\vert
  y \in G^r, \ \psi(y) = x
  \Bigr].
  \]
  Then the left-hand side of \eqref{eq:disc-pair} can be written as
  $\lvert \langle f - \tf, (u_1, \dots, u_r)^*_\psi
  \rangle\rvert$. Let
  $\cF$ be the set of functions which can be obtained by convex
  combinations of functions of the form $(u_1, \dots, u_r)^*_\psi$,
  varying over all combinations of functions $u_1, \dots, u_r \colon
  G^{r-1} \to [0,1]$ (but $\psi$ is fixed). Then $(f,\tf)$ being an
      $(r,\e')$-discrepancy pair with respect to $\psi$ is equivalent
      to $\lvert \langle f - \tf, \varphi
  \rangle\rvert \leq \e$ for all $\varphi \in \cF$. The desired
      claim would then follow from Lemma~\ref{lem:GTZ-DMT} and the triangle inequality provided
      we can show that $\cF$ is closed under multiplication. It suffices to show that for $u_1,
  \dots, u_r, u'_1, \dots, u'_r \colon G^{r-1} \to [0,1]$, the product
  of
  $(u_1, \dots, u_r)^*_\psi$ and $(u'_1, \dots, u'_r)^*_\psi$ still lies in
  $\cF$. Indeed, we have
  \begin{align*}
    (u_1, \dots, u_r)^*_\psi(x) (u'_1, \dots, u'_r)^*_\psi(x)
    &=   \EE\Bigl[ \prod_{i=1}^r u_i(y_{[r]\setminus\{i\}}) u'_i(y'_{[r]\setminus\{i\}})
  \Big\vert
  y,y' \in G^r, \ \psi(y) = \psi(y') = x
  \Bigr]
  \\
      &=   \EE\Bigl[ \prod_{i=1}^r u_i(y_{[r]\setminus\{i\}})
      u'_i(y_{[r]\setminus\{i\}} + z_{[r]\setminus\{i\}})
  \Big\vert
  y, z \in G^r, \ \psi(y) = x, \psi(z) = 0
  \Bigr]
  \\
      &=   \EE[(v_{1,z_{[r]\setminus \{1\}}},
      v_{2,z_{[r]\setminus\{2\}}}, \dots,
      v_{r,z_{[r]\setminus\{r\}}})_\psi^* (x) \vert z
      \in G^r, \ \psi(z) = 0]
  \end{align*}
  where $v_{i,z_{[r]\setminus\{i\}}} \colon G^{r-1} \to [0,1]$ is
  defined by $v_{i,z_{[r]\setminus\{i\}}}(y_{[r]\setminus\{i\}}) = u_i(y_{[r]\setminus\{i\}})
      u'_i(y_{[r]\setminus\{i\}} + z_{[r]\setminus\{i\}})$. This shows
      that the product of two such generalized convolutions is a
      convex combination of generalized convolutions, so that $\cF$ is
      closed under multiplication.
\end{proof}

\section{Counting lemma} \label{sec:count}

Next we show that if $(f,\tf)$ is a $(k-1,\e)$-discrepancy pair, with
$f \leq \nu$ and $\tf \leq 1$, then $f$ and $\tf$ have similar number of
(weighted) $k$-term APs. This is a special case of the counting lemma
for sparse hypergraphs from \cite{CFZrelsz}, whose self-contained
proof takes up about 4 pages \cite[Sec.~6]{CFZrelsz}.

\begin{lemma}[$k$-AP counting lemma]
  \label{lem:AP-counting}
  For every $k \geq 3$ and $\gamma > 0$, there exists an $\e > 0$ so
  that the following holds.

  Let $\nu, f, \tf \colon \ZZ_N \to \RR_{\geq 0}$ be
  functions. Suppose that $\nu$ satisfies the
  $k$-linear forms condition and $N$ is sufficiently large. Suppose
  also that $0 \leq f \leq \nu$, $0 \leq \tf \leq 1$, and $(f,\tf)$ is
  a $(k-1,\e)$-discrepancy pair with respect to each of $\psi_1, \dots,
  \psi_k$, where $\psi_j \colon \ZZ_N^{k-1} \to \ZZ_N$ is defined by
  \[
  \psi_j(x_1, \dots, x_{j-1}, x_{j+1}, \cdots, x_k) := \sum_{i \in
    [k]\setminus\{j\}} (i-j)x_i.
  \]
  Then
  \begin{equation}\label{eq:AP-counting}
    \Bigl\lvert \EE\Bigl[ \prod_{i=0}^{k-1} f(a +id) \Big\vert a,d \in
    \ZZ_N \Bigr]
    -
    \EE\Bigl[ \prod_{i=0}^{k-1} \tf(a +id) \Big\vert a,d \in
    \ZZ_N \Bigr]
    \Bigr\rvert
    \leq \gamma.
    \end{equation}
\end{lemma}

Let us explain why Lemma~\ref{lem:AP-counting} is a special case of \cite[Thm.~2.17]{CFZrelsz}.
We use the hypergraph notation from \cite[Sec.~2]{CFZrelsz}.
  Let $V = (J, (V_j)_{j\in J}, k-1, H)$ be a hypergraph system, where
  $J = [k]$, $V_j = \ZZ_N$ for every $j \in J$, and $H = \binom{J}{k-1}$ (corresponding to a simplex). Let
  $(\nu_e)_{e \in H}$, $(g_e)_{e \in H}$, and $(\tg_e)_{e \in H}$ be weighted hypergraphs on $V$ defined by
  \begin{align*}
  \nu_{[k]\setminus\{j\}}(x_{[k]\setminus\{j\}}) = \nu(\psi_j(x_{[k]\setminus\{j\}}))
\\
  g_{[k]\setminus\{j\}}(x_{[k]\setminus\{j\}}) = f(\psi_j(x_{[k]\setminus\{j\}}))
\\
  \tg_{[k]\setminus\{j\}}(x_{[k]\setminus\{j\}}) = \tf(\psi_j(x_{[k]\setminus\{j\}}))
  \end{align*}
  for $j \in [k]$ and $x_{[k]\setminus\{j\}} \in V_{[k]\setminus\{j\}} =
  \ZZ_N^{k-1}$. Then the weighted hypergraph $(\nu_e)_{e \in H}$ satisfies
  the $H$-linear forms condition \cite[Def.~2.8]{CFZrelsz} (which is equivalent to $\nu \colon \ZZ_N
  \to \RR_{\geq 0}$ satisfying
  the $k$-linear forms condition). That $(f,\tf)$ is a
  $(k-1,\e)$-discrepancy pair with respect to $\psi_j$ is equivalent to
  $(g_{[k]\setminus\{j\}},\tg_{[k]\setminus\{j\}})$ being an $\e$-discrepancy
  pair as  weighted hypergraphs \cite[Def.~2.13]{CFZrelsz}. Note that
\[
\EE\Bigl[ \prod_{i=0}^{k-1} f(a +id) \Big\vert x,d \in
    \ZZ_N \Bigr]
=
\EE\Bigl[ \prod_{e \in H} g_e(x_e) \Big\vert x \in V_J \Bigr]
\]
(to see this, let $a = \psi_1(x_2, \dots, x_k)$ and $d = - (x_1 + \cdots + x_k)$) and similarly with $\tf$ and $\tg_e$.
Then the relative hypergraph
  counting lemma \cite[Thm.~2.17]{CFZrelsz} reduces to Lemma~\ref{lem:AP-counting}.

\section{Proof of the relative Szemer\'edi theorem} \label{sec:relsz}

\begin{proof}[Proof of Theorem~\ref{thm:rel-sz}]
We begin with the
   following simple observation, that for any $g, g' \colon \ZZ_N \to
   \RR_{\geq 0}$, if $(g,g')$ is a $(k-1,\e)$-discrepancy pair with
   respect to one $\psi_j$ from Lemma~\ref{lem:AP-counting}, then it is so with respect to all
   $\psi_j$. This is simply because $1, 2, \dots, k-1$ all have
   multiplicative inverses in $\ZZ_N$, as $N$ is coprime to $(k-1)!$,
   and a scaling of variables in \eqref{eq:disc-pair} allows one to
   convert one linear form $\psi_j$ to another $\psi_{j'}$.

   The linear forms condition on $\nu$ implies that $(\nu, 1)$ is a
$(k-1,o(1))$-discrepancy pair with respect to $\psi_1$ from
Lemma~\ref{lem:AP-counting}. Indeed, we have the following inequality
\begin{equation}
  \label{eq:nu-1-disc}
  \Bigl\lvert \EE\Bigl[ (\nu(\psi_1(x)) - 1)\prod_{i=1}^r
  u_i(x_{[r]\setminus\{i\}}) \Big\vert x \in G^r \Bigr]\Bigr\rvert
  \leq
  \EE\Bigl[ \prod_{\omega\in\{0,1\}^r} (\nu(\psi_1(x^{(\omega)})) - 1)
  \Big\vert x^{(0)},x^{(1)} \in G^r \Bigr]^{1/2^r}
\end{equation}
which is proved by a sequence of Cauchy-Schwarz inequalities, similar
to \cite[Lem.~6.2]{CFZrelsz}. The right-hand side of \eqref{eq:nu-1-disc} is $o(1)$ by the linear forms condition (expand
the product so that each term is $\pm 1 + o(1)$ by \eqref{eq:Z-lfc}, and
everything cancels accordingly).

   Since $(\nu, 1)$ is a $(k-1,o(1))$-discrepancy pair with respect to
   $\psi_1$, Lemma~\ref{lem:disc-pair-transference} implies that
   there exists
$\tf \colon G \to [0,1]$ so that $\EE[\tf] = \EE[f] \geq \delta$ (if $\EE[f] > 1$, then replace $f$ by $\delta f / \EE[f]$)
   and $(f,\tf)$ is a $(k-1,o(1))$-discrepancy pair with respect to
   $\psi_1$, and hence with respect to all $\psi_j$, $1 \leq j \leq
   k$. So
   \begin{align*}
    \EE\Bigl[ \prod_{i=0}^{k-1} f(x +id) \Big\vert x,d \in
    \ZZ_N \Bigr]
    &\geq
    \EE\Bigl[ \prod_{i=0}^{k-1} \tf(x +id) \Big\vert x,d \in
    \ZZ_N \Bigr]
 - o(1)
 \\
    &\geq c(k,\delta) -o_{k,\delta}(1),
  \end{align*}
  where the first inequality is by Lemma~\ref{lem:AP-counting} and the
  second inequality is by Theorem~\ref{thm:sz-weighted}.
\end{proof}

\subsection*{Acknowledgments}
The author would like to thank Jacob Fox and David Conlon for careful readings of the manuscript.


\end{document}